\documentclass{amsart}

	\usepackage{aliascnt}
	\usepackage[colorlinks=true, linkcolor=blue, citecolor=green, urlcolor=blue]{hyperref}
	\usepackage{amssymb}	
	\usepackage{amsmath}
	\usepackage{amsthm}
	\usepackage{layout}
	\usepackage{tkz-euclide}
    \usepackage{tikz}
    \usepackage{algpseudocode,algorithm,algorithmicx}
    \usepackage{enumitem}
    \setenumerate[1]{label=\roman*)}
    
    \makeatletter
    \@namedef{subjclassname@2020}{\textup{2020} Mathematics Subject Classification}
    \makeatother

	\newtheorem{thm}{Theorem}[section]

	\newaliascnt{lemma}{thm}
	  
	\aliascntresetthe{lemma}

	\newaliascnt{prop}{thm}
	\newtheorem{prop}[prop]{Proposition} 
	\aliascntresetthe{prop}

	\newaliascnt{cor}{thm}
	
	\aliascntresetthe{cor}

	\theoremstyle{remark}

	\newaliascnt{rem}{thm}
	\newtheorem{rem}[rem]{Remark}
	\aliascntresetthe{rem}

	\theoremstyle{definition}

	\newaliascnt{exm}{thm}
	
	\aliascntresetthe{exm}

	\newaliascnt{notn}{thm}
	
	\aliascntresetthe{notn}

	\newaliascnt{defn}{thm}
	\newtheorem{defn}[defn]{Definition}
	\aliascntresetthe{defn}

    \newcommand{\Abl}{block of type $1$}
    \newcommand{\Abls}{blocks of type $1$}
    \newcommand{\Bbl}{block of type $2$}
    \newcommand{\Bbls}{blocks of type $2$}
	
\begin{document}

\title[Graphs of Hyperbinary Expansions]{\footnotesize Isomorphisms of graphs of Hyperbinary Expansions\\and Efficient Algorithms for Stern's Diatomic Sequence}
\author[Alessandro De Paris]{Alessandro De Paris}
\address{Link Campus University, Roma (Italy)}
\address{ORCID: \href{https://orcid.org/0000-0002-4619-8249}{0000-0002-4619-8249}}
\email{a.deparis@unilink.it}

\begin{abstract}
To investigate hyperbinary expansions of a nonnegative integer~$n$, an edge-labeled directed graph $A(n)$ has recently been introduced. After pointing out some new simple facts about its cyclomatic number, we give a relatively simple description of its structure and prove that if $m,n$ are even numbers for which $A(n)$ and $A(m)$ are isomorphic as edge-labeled graphs, then $m=n$. From the structure of $A(n)$ we also derive a formula related to Stern's diatomic sequence, and in the same vein discuss some algorithms that recently appeared in the literature.
\end{abstract}

\maketitle

\section{Introduction}

The present work pursues the line of research that began in \cite{BD} with the introduction of edge-labeled (or colored) directed graphs $A(n)$, whose vertices are the hyperbinary expansions of a given positive integer $n$. It is well known that the number $b(n)$ of hyperbinary expansions coincides with the term $c(n+1)$ of Stern's diatomic sequence. That is a much-studied topic with a long history and many features, for which we refer to \cite{N} as a good starting source of information. Other references, more specific to the subject of the present work, can be found in \cite[Introduction]{BD}, to which we add \cite{KS} because the techniques used there present some similarities with ours. A well-known striking feature, also reported in \cite[Introduction]{BD}, was the discovery in \cite{CW} that every rational number appears just once as a ratio between two subsequent terms in the sequence. The relationship between hyperbinary and binary signed-digit representations pointed out in \cite{M} also reveals its potential for computational purposes.

The last two sections of \cite{BD} are devoted to identify the integers that give rise to graphs with cyclomatic numbers zero and one, respectively. Along that line, to do the same for subsequent small integers may be of some interest, and \autoref{Ciclomatico} here is about that. In the subsequent section we prove our main result: graphs of hyperbinary expansions of two different even numbers cannot be isomorphic as edge-labeled graphs. Since the edge-labeled graph of each odd integer is isomorphic to the graph of an even number, we end up with a characterization of all pairs of integers with isomorphic edge-labeled graphs of hyperbinary expansions. That isomorphism theorem is proven by means of a detailed description of the structure of $A(n)$, which turns out to be relatively simple. We hope that it may shed new light on Stern's sequence. In the last section, from the structure of $A(n)$ we derive a formula for $b(n)$, more explicit than the well-known recursive one. We also turn it into an efficient algorithm, alternative to those in \cite{M}.

Before starting with the exposition of the results, we recapitulate here some basics. We shall keep $\Sigma^*$ as a standing notation for the free monoid generated by $0$, $1$, $2$. Its elements are just all finite sequences of these numbers and, as customary, they are named \emph{words} over the alphabet $\{0,1,2\}$ and are denoted by juxtaposition (that is, without parentheses and commas, though this way we lose the notation for the unit of $\Sigma$*, that is, the empty word). The multiplication in $\Sigma^*$ merges the words (again by juxtaposition).

When a word $x_0\ldots x_k\in\Sigma^*$ begins with $x_0\ne 0$, we say that it is a \emph{hyperbinary expansion} of the nonnegative integer $n:=\sum_{i=0}^kx_i2^{k-i}$. In conformity with \cite{BD}, we also assume \[\mathcal{H}(n)\qquad\text{and}\qquad b(n)\] as standing notation for the set of hyperbinary expansions of $n$ and its cardinality. But, in difformity with \cite{BD}, we allow $0\in\mathbb{N}$ ($\mathbb{N}_0$ is used there, instead, for the set of nonnegative integers) with the empty word as its unique hyperbinary expansion. This allows us not to treat it as a separate case like it is done instead in the presentation of the sequence $b$ at \cite[p.~27]{BD}.

Let us also recall here the definition of the labeled directed graph of hyperbinary expansions of a nonnegative integer $n$. To be consistent with \cite{BD}, the notation $A(n)$ will always indicate that graph. The set of arcs of a directed graph $G$ will be denoted by $\mathcal{E}(G)$ instead, and the set of vertices by $V(G)$. Let us recall that the two labels for $A(n)$ are conventionally denoted by $\to$ and $\twoheadrightarrow$. A \emph{single-step reduction of type $\to$} is a pair
\[
\left(\mathbf{x}02\mathbf{y},\mathbf{x}10\mathbf{y}\right)\quad\text{or}\quad\left(2\mathbf{y},10\mathbf{y}\right)
\]
with $\mathbf{x},\mathbf{y}\in\Sigma^*$; one \emph{of type $\twoheadrightarrow$} is a pair
\[
\left(\mathbf{x}12\mathbf{y},\mathbf{x}20\mathbf{y}\right)
\]
instead. The set of vertices of $A(n)$ is $\mathcal{H}(n)$ and the arcs are precisely all possible single step-reductions among them, accordingly labeled.

By a \emph{reduction} we mean a sequence of single-step reductions such that the child in each single-step reduction of the sequence, except the last, is the parent in the subsequent one (in the statement of \cite[Corollary~2.2, p.~29]{BD} they seem to use the term `reduction' only for the maximally lengthened ones). In graph-theoretic terms, they are directed paths in $A(n)$. These reductions induce a partial order on $A(n)$ that is compatible with the (total) shortlex order, as explained in \cite[Section~2]{BD}. The minimum and the maximum with respect to the shortlex order turn out to be also a global minimum and a global maximum with respect to the partial order, and are respectively the hyperbinary expansion without $0$s and the ordinary binary expansion of $n$ (see \cite[Corollary~2.2]{BD}). In the subsequent exposition, as in \cite{BD}, we shall often refer to the hyperbinary expansion without $0$s of $n$, as its \emph{minimal hyperbinary expansion}.

\section{On the cyclomatic number of \texorpdfstring{$A(n)$}{A(n)}}\label{Ciclomatico}

We begin with a description of the cyclomatic number of $A(n)$, for which we keep the notation $v(n)$ that is adopted in \cite{BD} (not to be confused with the set $V(G)$ of vertices of a graph $G$). The description will be solely in terms of $n$, that is, without reference to the structural data of $A(n)$ as it is done, instead, in \cite[Proposition~6.1]{BD}. To this end, let us first recall the well-known recursive formula for $b(n)$:
\begin{itemize}
\item $b(0)=1$;
\item $b(2n+1)=b(n)$;
\item $b(2n+2)=b(n+1)+b(n)$
\end{itemize}
(with $n\ge 0$; cf.~\cite[Proposition~2.2]{BD}). A similar formula for $v(n)$, but that involves $b(n)$ in addition, can be easily proved as follows.

\begin{prop}
The function $v(n)$ is determined by the following recursive conditions (where $n\ge 0$):
\begin{itemize}
\item $v(0)=0$;
\item $v(2n+1)=v(n)$;
\item $v(4n+2)=v(2n)+v(n)+b(n)-1$;
\item $v(4n+4)=v(2n+2)+v(n)+b(n)-1$.
\end{itemize}
\end{prop}
\begin{proof}
Since $A(0)$ is a simple graph with only one vertex, clearly $v(0)=0$. Since $A(2n+1)$ is isomorphic to $A(n)$ (see \cite[Proposition~3.2]{BD}), we also immediately have $v(2n+1)=v(n)$.

By \cite[Proposition~3.3]{BD} the set of vertices of $A(2m)$, with $m\ge 1$, can be partitioned into two sets such that the induced subgraphs are respectively isomorphic to $A(m-1)$ and $A(m)$. The tails of the remaining arcs (named \emph{bridging arcs} in \cite[p.~32]{BD}) are all the hyperbinary expansions of one of the forms $\mathbf{x}02$ or $\mathbf{x}12$, according to whether $m-1$ is even or odd, with $\mathbf{x}$ being a hyperbinary expansion of $\lfloor\frac{m-1}2\rfloor$.

Now, it is an elementary general fact that when a connected simple graph $G$ is partitioned into two nonempty induced subgraphs $G_1$ and $G_2$ that are connected, and $b$ is the (necessarily positive) number of remaining arcs, then the cyclomatic numbers are related by the formula
\[
v(G)=v\left(G_1\right)+v\left(G_2\right)+b-1\;.
\]

Taking $m=2n+1$ we have
\[
v(4n+2)=v(2n)+v(2n+1)+b(n)-1=v(2n)+v(n)+b(n)-1
\]
and taking $m=2n+2$ we have
\[
v(4n+4)=v(2n+1)+v(2n+2)+b(n)-1=v(2n+2)+v(n)+b(n)-1\;.
\]
\end{proof}

Since $A(n)$ is connected and nonempty, $v(n)+b(n)-1$ is the number $a(n)$ of arcs. Hence the last two recursive conditions may also be written in the simpler form
\begin{itemize}
\item $v(4n+2)=v(2n)+a(n)$;
\item $v(4n+4)=v(2n+2)+a(n)$.
\end{itemize}

In the same vein, one may also easily set up a recursive formula for $a(n)$, and gather some nice remarks about the functions $a$, $b$, $v$, which are invariants under isomorphisms of the graphs $A(n)$ (though they are dependent, because $A(n)$ is always connected). But here we prefer to stick to identifying those integers that give rise to graphs with given small cyclomatic numbers, pursuing the line of the last two sections of \cite{BD}. The recursive conditions for $v(n)$ suffice to this end.

First of all, the condition $v(2n+1)=v(n)$ allows one to reduce the problem to the case when $n$ is even. The conditions $v(4n+2)=v(2n)+a(n)$ and $v(4n+4)=v(2n+2)+a(n)$ connect the value of $v$ on any positive even number with the value on smaller even numbers, and the latter (value) cannot be bigger. In particular, if $v(4n+2)=0$ or $v(4n+4)=0$, then $a(n)=0$ and this is possible only when $n$ has only one hyperbinary expansion, that contains only $1$s (otherwise a reduction can occur somewhere); that is, $n=2^t-1$, with $t$ being a nonnegative integer. This shows that a positive even number on which $v$ vanishes, must be of the form $2^{t+2}-2$ or $2^{t+2}$. But, again by the recursive conditions, for $t=0$ we have $v(2)=v(0)+a(0)=0$, $v(4)=v(2)+a(0)=0$ and for all $t$, $v\left(2^{t+3}-2\right)=v\left(2^{t+2}-2\right)+a\left(2^{t+1}-1\right)$ and $v\left(2^{t+3}\right)=v\left(2^{t+2}\right)+a\left(2^{t+1}-1\right)$. Hence, by induction and including $v(0)=0$ we have that the even numbers $n$ such that $v(n)=0$ are exactly those of the form $2^t-2$ or $2^t$, with $t$ a positive integer. For arbitrary positive integers $n$, the recursive condition $v(2n+1)=v(n)$ easily leads to the content of \cite[Theorem~6.2]{BD} (without $>0$, as here we are allowing $n=0$).

Let us now suppose that $v(4n+2)=1$, and take again into account $v(4n+2)=v(2n)+a(n)$. It cannot be $a(n)=0$, otherwise by the previous analysis we would have $v(4n+2)=0$. Hence $v(2n)=0$ and $a(n)=1$. Again by the previous analysis, $v(2n)=0$ implies that $n=2^t-1$ or $n=2^t$, with $t$ being a nonnegative integer. But $n=2^t-1$ and $n=2^0$ are excluded because lead to $a(n)=0$, and since for $t\ge 1$ the expansion $1^{t-1}2\in\Sigma^\ast$ of $n=2^t$ admits a reduction of length $t$, from $a(n)=1$ we deduce that $t=1$ and $n=2$; therefore $4n+2=10$. Similar arguments show that from $v(4n+4)=1$ also follows $n=2$, and therefore $4n+4=12$. This shows that the even numbers $n$ for which $v(n)=1$ are only $10$ and $12$, which is equivalent to the last statement in \cite[Section~7]{BD}, and (as is also proved therein) from the recursive condition $v(2n+1)=v(n)$ can be immediately deduced that the numbers $n$ with $v(n)=1$ are exactly those of the form $(12\pm 1)2^t-1$, with $t$ being a nonnegative integer.

To find out the even numbers on which $v$ takes value $2$, we have to examine the cases
\begin{itemize}
\item $v(2n)=1$ and $a(n)=1$;
\item $v(2n+2)=1$ and $a(n)=1$;
\item $v(2n)=0$ and $a(n)=2$;
\item $v(2n+2)=0$ and $a(n)=2$.
\end{itemize}
We know that on even numbers $v$ takes value $1$ only on $10$ and $12$, and taking into account that $a(4)=2$, $a(5)=1$ and $a(6)=2$, the first two cases give $n=5$ and the values $4\cdot 5+2=22$ and $4\cdot 5+4=24$. Proceeding like before, the third case leads to $n=4$, which gives the value $4\cdot 4+2=18$. The fourth case leads to $n=6$, which gives $4\cdot 6+4=28$. Hence 
\[
n\text{ even }\,\land\,v(n)=2\;\,\iff\;\,n\in\{18,22,24,28\}
\]
which is the first step beyond \cite{BD}. Another step ahead is
\[
n\text{ even }\,\land\,v(n)=3\;\,\iff\;\,n\in\{20,26,34,46,48,60\}\;,
\]
for which we omit the now straightforward details.

Instead of investigating these sets for every assigned cyclomatic number, now we prefer to focus on two goals that can be addressed by means of them, which we believe are somewhat suggested in \cite{BD}. The first one is to understand to what extent the labeled directed graph structures characterize the numbers they come from (though the invariants $b(n)$ and $v(n)$ are not affected by the labeling and the directions of the arcs, they may be helpful in that respect). Next, the results about the case $v(n)=1$ in \cite[Section~7]{BD} are worked out in connection with the blocks of $1$'s and $2$'s in the minimal hyperbinary expansion of $n$. A way to calculate $v(n)$, and also $b(n)$, from the number and the lengths of those blocks may be interesting not only in view of the first goal, but also for the computational aspects of related topics.

In the following section we show that a slightly different block decomposition of the minimal hyperbinary extension of $n$ gives a lot of information on the structure of $A(n)$, that allows to prove with moderate effort that even numbers with isomorphic labeled directed graphs of hyperbinary expansions must coincide (which in turn immediately gives a characterization for whatever numbers with such isomorphic graphs, because of the isomorphims between $A(n)$ and $A(2n+1)$). The results will also shed light on our previous calculations (and that is why we stopped them at $v(n)=3$). In the last section we derive a formula for $b(n)$ in terms of the block decomposition.

As anticipated, our blocks will be slightly different from the mere sequences of $1$s and of $2$s in a given minimal hyperbinary expansion. To introduce them, let us first point out that the trees of \cite[Section~6]{BD} have no branches, that is (being also connected), are directed path graphs. Indeed, since (by \cite[Corollary~2.2 (i)]{BD}) each reduction can be lengthened to reach the sink of $A(n)$, that is, the ordinary binary expansion of $n$, we have that any vertex with outdegree $\ge 2$ gives rise to at least one (non-directed) cycle. Our blocks will be precisely the minimal hyperbinary expansions of even numbers~$n$ with $v(n)=0$, and their path graphs $A(n)$ in some sense will `span' the graphs of hyperbinary expansions of all numbers.

\section{Isomorphism theorem}

\begin{defn}
For any positive integer $t$, let us name \emph{\Abl} the element $1^t2$ of the monoid $\Sigma^\ast$, that is, $1\cdots 12$, and \emph{\Bbl} the element $2^t=2\cdots 2$.
\end{defn}
Plainly, any minimal hyperbinary expansion that ends with $2$, that is, that of a positive even number, can be written in a unique way as a product of such blocks with no consecutive \Bbls. Minimal hyperbinary expansions of odd numbers are in a unique way products in $\Sigma^*$ of such a product of blocks and a power $1^t=1\cdots 1$ ($t\ge 1$). We refer to this decomposition as the \emph{block decomposition} of minimal hyperbinary expansions.

If the minimal hyperbinary expansion of $n$ is a \Abl, then $A(n)$ is the directed path graph
\begin{equation}\label{Path1}
1^t2\twoheadrightarrow 1^{t-1}20\twoheadrightarrow\cdots\twoheadrightarrow 20^t\rightarrow 10^{t+1}\;,
\end{equation}
whereas for an $n$ whose minimal hyperbinary expansion is a \Bbl, $A(n)$ is the directed path graph
\begin{equation}\label{Path2}
2^t\to 102^{t-1}\to\cdots\to 1^t0\;.
\end{equation}
When we have a minimal hyperbinary expansion made of (exactly) two blocks, each step in any reduction occurs on either one of the two parts (i.e., factors in $\Sigma^\ast$) that come from the original two blocks, according to \eqref{Path1} or \eqref{Path2}. Here we say that a step occurs `on a part', when that part include the last digit in the parent that is modified in the child, and we shall soon explain more precisely the meaning of `come from'. Those single-step reductions can be performed almost freely, with only the following constraints:
\begin{enumerate}
\item\label{Constr12} when we have a \Abl\ followed by a \Bbl, the first single-step reduction on the second block can occur only after that the first single-step reduction has occurred on the first block;
\item\label{Constr11} when we have two consecutive \Abls, the last single-step reduction on the expansion coming from reduction of the second block can occur only after that the first single-step reduction has occurred on the first block;
\item\label{Constr21} when we have a \Bbl\ followed by a \Abl, the last single-step reduction on the expansion coming from reduction of the second block can occur only after that the last single-step reduction has occurred on the expansion coming from reduction of the first block;
\end{enumerate}

In the above description of the constraints, by a part that `come from' reduction of a block we simply mean that it is a vertex in the directed path graphs \eqref{Path1} or \eqref{Path2}, multiplied by the other part of the whole expansion. But one has to pay attention to the fact that, once the `constrained' single-step reduction on the second part has occurred, the first part coincide with an expansion in \eqref{Path1} or \eqref{Path2}, except for the last digit, which is turned from $0$ to $1$ (and is overlapped by the second part). This do not give rise to substantially new single-step reductions on the first part, other than those in \eqref{Path1} or \eqref{Path2}: after that the mentioned `unlocked' single-step reduction has occured in the second part, it suffices to truncate them by removing the last digit. This explains the meaning of the `coming' of the two parts from the original two blocks.

\begin{rem}\label{CG}
The above fact can be expressed graph-theoretically and, more generally, for whatever product $\mathbf{n}=\mathbf{n}_1\mathbf{n}_2$ of minimal hyperbinary expansions of $n,n_1,n_2$ with $n_1$, $n_2$ and (henceforth) $n$ even. For, let us consider the product $\mathcal{H}\left(n_1\right)\times\mathcal{H}\left(n_2\right)$ and map a subset $H$ of it into $\mathcal{H}(n)$. Pairs $\left(\mathbf{n}''_1,\mathbf{n}'_2\right)$ in the subset $H$ are exactly those such that $\mathbf{n}''_1$ ends with $0$ or $\mathbf{n}'_2$ is a short or empty expansion (see \cite[Sect.~5]{BD} for the notion of short and long expansions of a positive integer). The image of $\left(\mathbf{n}''_1,\mathbf{n}'_2\right)$ is the product $\mathbf{n}'_1\mathbf{n}'_2$ (in $\Sigma^\ast$), with $\mathbf{n}'_1$ being either $\mathbf{n}''_1$, when $\mathbf{n}'_2$ is short or empty, or its truncation on the last digit, when $\mathbf{n}'_2$ is long.

By our previous discussion and the fact that every $\mathbf{n}'\in\mathcal{H}(n)$ comes from reduction of $\mathbf{n}$ (by \cite[Corollary~2.2 (ii)]{BD}), we have that the map $H\to\mathcal{H}(n)$ is surjective. Since the construction can easily be inverted, we conclude that the map is bijective.
\end{rem}

Our previous description of the single-step reductions can easily be cast in these terms as well. Let us recall that a Cartesian product (or box product) $G_0\,\Box\, G_1$ of simple directed graphs has $V\left(G_0\right)\times V\left(G_1\right)$ as its set of vertices, an arc $\big(\,\left(x_0,x_1\right)\,,\,\left(y_0,x_1\right)\,\big)$ for each choice of an arc $\left(x_0,y_0\right)$ of $G_0$ and a vertex $x_1$ of $G_1$, and an arc $\big(\,\left(x_0,x_1\right)\,,\,\left(x_0,y_1\right)\,\big)$ for each choice of a vertex $x_0$ of $G$ and an arc $\left(x_1,y_1\right)$ of $G_1$. Plainly, the map $H\to\mathcal{H}(n)$ in \autoref{CG} gives an isomorphism between the subgraph induced by $A\left(n_1\right)\,\Box\, A\left(n_1\right)$ on $H$, and $A(n)$. Finally, reversing the map $H\to\mathcal{H}(n)$, we can summarize our previous speech by saying that $A(n)$ canonically embeds, as an induced subgraph, into the Cartesian product $A\left(n_1\right)\,\Box\, A\left(n_2\right)$. Let us also note that each arc of a product involve, in particular, a choice of one of the two factors and an arc in it. Hence arc-labeling on factors induce arc-labeling on the product, and in the present case the induced labeling on $A\left(n_1\right)\,\Box\, A\left(n_2\right)$ makes the morphism that embeds $A(n)$ into it a morphism of arc-labeled directed graphs.

The picture below illustrates the above situation for all block decompositions made of two blocks of the shortest length. Here we dismiss the rule introduced in \cite[p.~30]{BD}, by which the vertical alignment of each vertex in the picture reflects its weight $\omega$ (that is, the sum of its digits, which actually also coincide with the function $T$ introduced in \cite[Sect.~4]{BD} in connection with earlier works on the subject).

\noindent
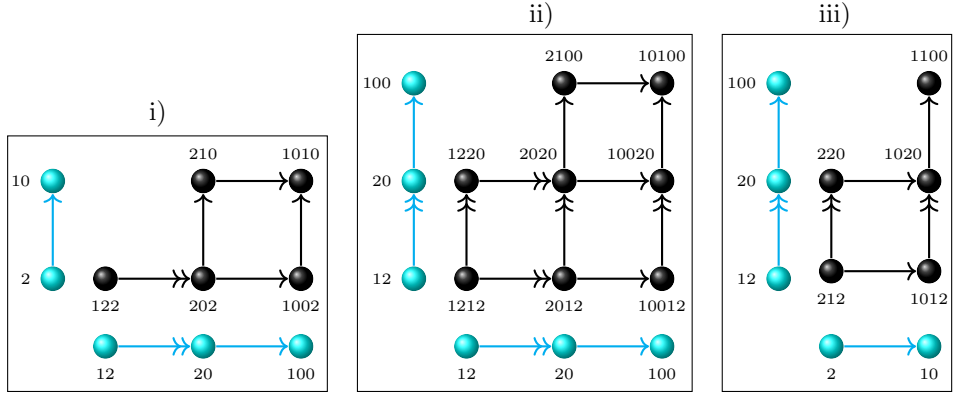
\begin{tikzpicture}[c/.style={{circle, shading=ball, ball color=cyan}}, b/.style={{circle, shading=ball, ball color=black}}]
\node (i) at (0.7,3.7) [label]{i)};
 \draw [draw=black] (-1.3,0) rectangle (2.95,3.4);
 \node[c] (Ax2)  at (0,.6)  [label=below:\tiny $12$]{};
 \node[c] (Bx2)  at (1.3,.6)  [label=below:\tiny $20$]{} ;
 \node[c] (Cx2)  at (2.6,.6)  [label=below:\tiny $100$]{} ;
 \node[c] (Ay2)  at (-.7,1.5)  [label=left:\tiny $2$]{} ;
 \node[c] (By2)  at (-.7,2.8)  [label=left:\tiny $10$]{} ;
 \node[b] (AA2)  at (0,1.5)  [label=below:\tiny $122$]{} ;
 \node[b] (BA2)  at (1.3,1.5)  [label=below:\tiny $202$]{} ;
 \node[b] (CA2)  at (2.6,1.5)  [label=below:\tiny $1002$]{} ;
 \node[b] (BB2)  at (1.3,2.8)  [label=above:\tiny $210$]{} ;
 \node[b] (CB2)  at (2.6,2.8)  [label=above:\tiny $1010$]{} ;
 \path [->>, cyan, thick] (Ax2) edge (Bx2);
 \path [->, cyan, thick] (Bx2) edge (Cx2);
 \path [->, cyan, thick] (Ay2) edge (By2);
 \path [->>, thick] (AA2) edge (BA2);
 \path [->, thick] (BA2) edge (BB2);
 \path [->, thick] (BB2) edge (CB2);
 \path [->, thick] (BA2) edge (CA2);
 \path [->, thick] (CA2) edge  (CB2);
 \node (ii) at (5.8,5) [label]{ii)};
 \draw [draw=black] (3.35,0) rectangle (7.8,4.75);
 \node[c] (Ax)  at (4.8,.6)  [label=below:\tiny $12$]{} ;
 \node[c] (Bx)  at (6.1,.6)  [label=below:\tiny $20$]{} ;
 \node[c] (Cx)  at (7.4,.6)  [label=below:\tiny $100$]{} ;
 \node[c] (Ay)  at (4.1,1.5)  [label=left:\tiny $12$]{} ;
 \node[c] (By)  at (4.1,2.8)  [label=left:\tiny $20$]{} ;
 \node[c] (Cy)  at (4.1,4.1)  [label=left:\tiny $100$]{} ;
 \node[b] (AA)  at (4.8,1.5)  [label=below:\tiny $1212$]{} ;
 \node[b] (BA)  at (6.1,1.5)  [label=below:\tiny $2012$]{} ;
 \node[b] (CA)  at (7.4,1.5)  [label=below:\tiny$10012$]{} ;
 \node[b] (AB)  at (4.8,2.8)  [label=above:\tiny $1220$]{} ;
 \node[b] (BB)  at (6.1,2.8)  [label=above:\tiny\hspace{-22pt} $2020$]{} ;
 \node[b] (CB)  at (7.4,2.8)  [label=above:\tiny\hspace{-26pt} $10020$]{} ;
 \node[b] (BC)  at (6.1,4.1)  [label=above:\tiny $2100$]{} ;
 \node[b] (CC)  at (7.4,4.1)  [label=above:\tiny $10100$]{} ;
 \path [->>, cyan, thick] (Ax) edge (Bx);
 \path [->, cyan, thick] (Bx) edge (Cx);
 \path [->>, cyan, thick] (Ay) edge (By);
 \path [->, cyan, thick] (By) edge (Cy);
 \path [->>, thick] (AA) edge (BA);
 \path [->>, thick] (BA) edge (BB);
 \path [->>, thick] (AA) edge (AB);
 \path [->>, thick] (AB) edge (BB);
 \path [->, thick] (BB) edge (CB);
 \path [->, thick] (CB) edge (CC);
 \path [->, thick] (BB) edge (BC);
 \path [->, thick] (BC) edge (CC);
 \path [->, thick] (BA) edge (CA);
 \path [->>, thick] (CA) edge (CB);
\node (iii) at (9.7,5) [label]{iii)};
\draw [draw=black] (8.2,0) rectangle (11.25,4.75);
\node[c] (Ax1)  at (9.65,.6)  [label=below:\tiny $2$]{} ;
 \node[c] (Bx1)  at (10.95,.6)  [label=below:\tiny $10$]{} ;
 \node[c] (Ay1)  at (8.95,1.5)  [label=left:\tiny $12$]{} ;
 \node[c] (By1)  at (8.95,2.8)  [label=left:\tiny $20$]{} ;
 \node[c] (Cy1)  at (8.95,4.1)  [label=left:\tiny $100$]{} ;
 \node[b] (AA1)  at (9.65,1.6)  [label=below:\tiny $212$]{} ;
 \node[b] (BA1)  at (10.95,1.6)  [label=below:\tiny $1012$]{} ;
 \node[b] (AB1)  at (9.65,2.8)  [label=above:\tiny $220$]{} ;
 \node[b] (BB1)  at (10.95,2.8)  [label=above:\tiny \hspace{-22pt} $1020$]{} ;
 \node[b] (BC1)  at (10.95,4.1)  [label=above:\tiny $1100$]{} ;
 \path [->, cyan, thick] (Ax1) edge (Bx1);
 \path [->>, cyan, thick] (Ay1) edge (By1);
 \path [->, cyan, thick] (By1) edge (Cy1);
 \path [->, thick] (AA1) edge (BA1);
 \path [->, thick] (AB1) edge (BB1);
 \path [->, thick] (BB1) edge (BC1);
 \path [->>, thick] (AA1) edge (AB1);
 \path [->>, thick] (BA1) edge (BB1);
\end{tikzpicture}

The above said plainly extends to block decompositions with more than two factors (of odd numbers too, with the final sequence of $1$s being irrelevant). So, let us consider a block decomposition \[\mathbf{n}=\mathbf{n}_1\cdots\mathbf{n}_r\] of the minimal hyperbinary expansion of an even number $n$. In terms of reductions, we have that every hyperbinary expansion $\mathbf{n}'$ of $n$ can be decomposed in $\Sigma^*$ in exactly one way as a product
\[
\mathbf{n}'=\mathbf{n}'_1\cdots\mathbf{n}'_r
\]
of nonempty factors, with each factor $\mathbf{n}'_i$ coming from reduction of $\mathbf{n}_i$, according to \eqref{Path1} or \eqref{Path2}. In a more technical and complete form, we end up with the following proposition.

\begin{prop}
Let $\mathbf{n}=\mathbf{n}_1\cdots \mathbf{n}_r$ be the block decomposition of the minimal hyperbinary expansion of an even number $n$, and let $n_1$, $\ldots$, $n_r$ be the (even) numbers with respective expansions $\mathbf{n}_1$, $\ldots$, $\mathbf{n}_r$. Then there is a uniquely determined morphism of labeled directed graphs that maps $A(n)$ isomorphically onto an induced subgraph of the Cartesian product
\[
A\left(n_1\right)\,\Box\,\cdots\,\Box\, A\left(n_r\right)\;
\]
such that for each vertex $\mathbf{n}'$ of $A(n)$, denoting by
\[
\left(\mathbf{n}''_1,\ldots,\mathbf{n}''_r\right)
\]
its image, we have that
\[
\mathbf{n}'=\mathbf{n}'_1\cdots \mathbf{n}'_r\;,
\]
with $\mathbf{n}'_r:=\mathbf{n}''_r$ and $\mathbf{n}'_i$ being, for each $i\in\{1,\ldots, r-1\}$, either $\mathbf{n}''_i$ or its truncation on the last digit, according to whether $\mathbf{n}''_{i+1}$ is a short or a long expansion. 
\end{prop}

Like in the two factors case, each arc of the Cartesian product
\[
G:=\;\Box_{i\in I}\,G_i
\]
of a family of directed graphs, involve in particular a choice of one of the factors (that is, of an index $i\in I$) and an arc in it. This gives in particular a map $\mathcal{E}(G)\to I$. Since a graph $A(n)$ naturally embeds into a product, we also have an induced map \[\pi:\mathcal{E}(A(n))\to\{1,\ldots,r\}\;,\] with $r$ being the number of blocks in the minimal hyperbinary decomposition of $n$, which we shall call the \emph{place map} of $A(n)$. The meaning is clear: an arc of $A(n)$ is a single-step reduction of a hyperbinary expansion of $n$, and since hyperbinary expansion of $n$ are all split into $r$ factors, the place map associates the arc with the place among $1,\ldots,r$ where the single-step reduction occurs. In this description, single-step reductions `at the hinge' of two factors (necessarily of type $\to$) are considered as occurring on the second one.

The property that is pointed out in the following remark clearly holds for whatever Cartesian product of directed path graphs. It does hold as well for the subgraphs $A(n)$, basically because a single-step reduction that may occur on a given expansion, cannot be `locked' by another single step reduction.

\begin{rem}\label{Inv}
Let $e:=(x,y)$ be an arc of $A(n)$ and let $\mathcal{E}^+(x)$, $\mathcal{E}^+(y)$ be the set of arcs of $A(n)$ from $x$ and from $y$, respectively. For each arc $e_x:=\left(x,x'\right)\in\mathcal{E}^+(x)\smallsetminus\{e\}$ there exists exactly one arc $e_y:=\left(y,y'\right)\in\mathcal{E}^+(y)$ such that $\left(x',y'\right)$ is an arc of $A(n)$.
\end{rem}

\begin{defn}
In notation of \autoref{Inv}, we define the \emph{place-preserving map through $e$} as the map 
\[
\mathcal{E}^+(x)\smallsetminus\{e\}\to\mathcal{E}^+(y)
\]
that sends each $e_x$ to $e_y$.

Given a directed path $\left(e_1,\ldots,e_t\right)$ of $A(n)$, we define more generally the \emph{place-preserving map through $\left(e_1,\ldots,e_t\right)$} the composition of (maximally possible restrictions of) $\alpha_{e_t},\ldots,\alpha_{e_1}$, with $\alpha_{e_i}$ being the place-preserving map through $e_i$, for each $i\in\{1,\ldots,t\}$.
\end{defn}

\begin{rem}
If $\pi:\mathcal{E}(A(n))\to\{1,\ldots,r\}$ is the place map of $A(n)$ and $\alpha:X\to\mathcal{E}^+(y)$ (with $X\subset\mathcal{E}^+(x)$) is a place-preserving map through an arc or a directed path, we have $\pi\circ\alpha=\pi$.
\end{rem}
In other words, the place of any arc from the starting vertex such that the place-preserving map is defined, is the same as the place of its image.

\begin{rem}
Suppose that there exists a directed graph isomorphism $\varphi:A\left(n_0\right)\to A\left(n_1\right)$ between graphs of hyperbinary expansions, and let $\alpha_0:\mathcal{E}^+\left(x_0\right)\smallsetminus\left\{e_0\right\}\to\mathcal{E}^+\left(y_0\right)$, $\alpha_1:\mathcal{E}^+\left(x_1\right)\smallsetminus\left\{e_1\right\}\to\mathcal{E}^+\left(y_1\right)$ be place-preserving maps, respectively through arcs $e_0\in\mathcal{E} \left(A\left(n_0\right)\right)$, $e_1\in\mathcal{E}\left(A\left(n_1\right)\right)$ that correspond to each other through~$\varphi$.

Since the construction in \autoref{Inv} involve only the graph structure, for every arc $e_{x_0}\in\mathcal{E}^+\left(x_0\right)\smallsetminus\left\{e_0\right\}$, denoting by $e_{x_1}\in\mathcal{E}^+\left(x_1\right)\smallsetminus\left\{e_1\right\}$ the arc that corresponds to it through the isomorphism $\varphi$, we have that 
\[
\alpha_0\left(e_{x_0}\right)\qquad\text{and}\qquad\alpha_1\left(e_{x_1}\right)
\]
also correspond to each other through $\varphi$.

Clearly a similar statement holds if one considers directed paths (that correspond through the isomorphism $\varphi$), instead of the arcs $e_0,e_1$.
\end{rem}

Our proof of the isomorphism theorem will rely on the possibility of reconstructing the block decomposition from the labeled structure of $A(n)$. To this end, we need to introduce the following notion.

\begin{defn}
Let us call a \emph{checking path in $A(n)$} a directed path $\left(e_1,\ldots,e_r\right)$ such that for all $i\in\{1,\ldots,r-1\}$ we have that $e_{i+1}$ is not in the image of the place-preserving map through $e_i$.

We also say that such a checking path is \emph{maximal} when there is no arc $e$ such that $\left(e,e_1,\ldots,e_r\right)$ or $\left(e_1,\ldots,e_r,e\right)$ is a checking path.
\end{defn}

Typical checking paths of $A(n)$ are made of reductions that all occur at a given place (for instance, $122\twoheadrightarrow 202\to 1002$ in $A(10)$). That is because no two different single-step reductions of the same expansion can occur on the same place, hence each arc after the first one lies outside the place-preserving map through the preceding one (which has the same place and is outside the domain of the map). But `unlocked' steps may give rise to exceptional checking paths. For instance, in the reduction
\[
122\twoheadrightarrow 202\to 210\;,
\]
the second step occurs on the second factor (actually, at the hinge), but no single-step reduction of $122$ can correspond to it through the place-preserving map through the first step, because `it is locked' by the first block $12$ (cf.~the constraint n.~\ref{Constr12}). Thus that is a checking path as well. Of course when directed paths correspond to each other trough a graph isomorphism, we have that one of them is a checking path if and only if so is the other, and henceforth also that one of them is maximal if and only if the other is maximal.

\begin{thm}
Let $m,n$ be even nonnegative integers. If $A(m)$ and $A(n)$ are isomorphic as edge-labeled directed graphs, then $m=n$.
\end{thm}
\begin{proof}
Let $\mathbf{n}=\mathbf{n}_1\cdots\mathbf{n}_r$ and $\mathbf{m}=\mathbf{m}_1\cdots\mathbf{m}_s$ be the block decompositions of the minimal hyperbinary expansions of $n$ and $m$, respectively. We prove the theorem by induction on $r$, by showing that $r=s$ and $\mathbf{n}_i=\mathbf{m}_i$ for all $i\in\{1, \ldots, r\}$. When $r=0$, then $n=0$, $A(n)$ has only one vertex, the same must hold for the isomorphic graph $A(m)$, hence $s=0$ as well (otherwise some reduction in $A(m)$ would exist). Let us suppose, then, that $r\ge 1$ and that the assertion is true when in place of $n,m$ there are even numbers with isomorphic graph of hyperbinary expansions and one of them comes with a block decomposition consisting of $r-1$ blocks. Let us also fix an isomorphism $\varphi:A(n)\to A(m)$.

Let us suppose first that there exist an arc $e_1$ from the source $\mathbf{n}$ of $A(n)$, that is labeled $\to$. Since the first reduction on blocks of type $1$ is not of that type, $e$ (as a reduction) must occur on a block of type $2$, say of length $t$ (as a word in $\Sigma^*$). Due to the constraint n.~\ref{Constr12}, that block must be $\mathbf{n}_1$. Moreover, let us consider the path $\left(e_1,\ldots, e_t\right)$ of $A(n)$ that starts with $e_1$ and corresponds to the path \eqref{Path2} of $A(\mathbf{n}_1)$. All reductions in that path occur on the first place (that is, the value of the place map on them is always $1$). Hence $\left(e_1,\ldots, e_t\right)$ is a checking path. Moreover, any arc that is not in the path but whose tail is a vertex in the path, must correspond to first reductions on blocks in other places. Since $\mathbf{n}_1$ is a block of type $2$, $\mathbf{n}_2$ must be of type $1$, and then, for each vertex of the path, the constraint n.~\ref{Constr21} do not lock the first single-step reduction that occurs on the (possible) factor $\mathbf{n}_2$, nor of course first single-step reductions that occur on the other blocks that (possibly) follow. In other words, for each $i\in\{1,\ldots,t-1\}$, no vertex from the the head of $e_i$, other than $e_{t+1}$, is outside the image of the place-preserving map through $e_i$. This shows that $\left(e_1,\ldots, e_t\right)$ is the only possible checking path with length $t$ that starts with $e_1$. Clearly the last vertex is the expansion $\mathbf{n}':=\mathbf{n}'_1\mathbf{n}_2\cdots\mathbf{n}_r$, with $\mathbf{n}'_1$ being the ordinary binary expansion of the number with expansion $\mathbf{n}_1$. It easily follows that the place-preserving map through $e_t$ is surjective, and then $\left(e_1,\ldots, e_t\right)$ is the unique maximal checking path that starts with $e_1$.

Now, the arc $f_1$ of $A(m)$ that corresponds to $e_1$ through the isomorphism $\varphi$ must be an arc labeled $\to$ from the source $\mathbf{m}$ of $A(m)$. The arguments above apply to $f_1$ as well, and show that the block $\mathbf{m}_1$ is of type $2$ and there exists exactly one maximal checking path in $A(m)$ that
\begin{itemize}
\item has $\mathbf{m}$ as starting vertex,
\item starts with an arc labeled $\to$,
\item has length equal to the length of $\mathbf{m}_1$,
\item its ending vertex is the expansion $\mathbf{m}':=\mathbf{m}'_1\mathbf{m}_2\cdots\mathbf{m}_s$, with $\mathbf{m}'_1$ being the ordinary binary expansion of the number with expansion $\mathbf{m}_1$.
\end{itemize}
Clearly, the maximal checking paths we have found in $A(n)$ and in $A(m)$ must correspond to each other through $\varphi$. Thus we have that
\begin{description}
\item[S1]\label{S1} $\mathbf{n}_1$ and $\mathbf{m}_1$ are blocks of the same type and of the same length;
\item[S2]\label{S2} the expansions $\mathbf{n}':=\mathbf{n}'_1\mathbf{n}_2\cdots\mathbf{n}_r$ and $\mathbf{m}':=\mathbf{m}'_1\mathbf{m}_2\cdots\mathbf{m}_s$, respectively of $n$ and $m$, with $\mathbf{n}'_1$ and $\mathbf{m}'_1$ being ordinary binary expansions (of some numbers), correspond to each other through the isomorphism $\varphi$.
\end{description}
It is not difficult to see that the subgraph of $A(n)$ induced by the set made of $\mathbf{n}'$ and all its descendants (in other words, the vertices that can be reached from $\mathbf{n}'$ by a directed path) is isomorphic to $A\left(\overline{n}\right)$, with $\overline{n}$ being the number with expansion $\mathbf{n}_2\cdots\mathbf{n}_r$. By the statement \textbf{S2} above, through a restriction of $\varphi$, that subgraph is clearly also isomorphic to the subgraph of $A(m)$ induced by the set made of  $\mathbf{m}'$ and all its descendants, which in turn is isomorphic to the graph of hyperbinary expansions of the number with expansion $\mathbf{m}_2\cdots\mathbf{m}_s$. By the induction hypothesis, we have that $r-1=s-1$ and $\mathbf{n}_i=\mathbf{m}_i$ for each $i\in\{2,\ldots,r\}$. Hence $r=s$ and, by the statement \textbf{S1} above, obviously $\mathbf{n}_1=\mathbf{m}_1$ as well. This proves the theorem in the case when there exists an arc $e_1$ from the source $\mathbf{n}$ of $A(n)$ that is labeled $\to$. Let us note that the final deduction from the statements~\textbf{S1} and~\textbf{S2} above, do not rely on that assumption. Hence, at this point, to prove those statements without that assumption will suffice.

So, let $e_1$ be an arc from the source $\mathbf{n}$ of $A(n)$ that, this time, is labeled $\twoheadrightarrow$. Then, as a reduction, $e_1$ must occur on a block $\mathbf{n}_a$ of type $1$, and this time let $t$ denote its length (necessarily $\ge 2$). The path \eqref{Path1} gives rise to a path in $A(n)$ starting with $e_1$ (made of reductions that all occur on the $a$th place), but except for the last arc when $a\ne 1$ (due to the constraints n.~\ref{Constr11} and n.~\ref{Constr21}). Hence that path can be denoted by $\left(e_1,\ldots,e_{t'}\right)$ with $t'=t$ if $a=1$ and $t'=t-1$ otherwise, and all its arcs are labeled $\twoheadrightarrow$, with the possible exception of $e_t$ (when it actually appears). Among the vertices of $\left(e_1,\ldots,e_{t'}\right)$, the constraint n.~\ref{Constr12} may lock the first-step reduction on $\mathbf{n}_{a+1}$ only at the first one, and only when $\mathbf{n}_{a+1}$ (does exist and) is of type $2$. In that special situation, there is also a checking path $\left(e_1,e'_2\right)$ with $e'_2$ labeled $\to$ and different from $e_2$. On the contrary, when there is no $\mathbf{n}_{a+1}$ of type $2$, $\left(e_1,\ldots,e_{t'}\right)$ is the unique maximal checking path that starts with~$e_1$.

The above analysis implies that when $a=1$ we have:
\begin{itemize}
\item if $t\ge 3$ then $\left(e_1,\ldots,e_t\right)$ is the unique maximal checking path with length $t$, starting vertex $\mathbf{n}$, $e_2$ labeled $\twoheadrightarrow$ and $e_t$ labeled $\to$;
\item if $t=2$ there is no maximal checking path in $A(n)$ with starting vertex $\mathbf{n}$ such that its second arc is labeled $\twoheadrightarrow$ and its last arc is labeled $\to$.
\end{itemize}
This in turn implies that if $A(n)$ has a maximal checking path of length $t\ge 3$, with starting vertex $\mathbf{n}$ and such that its second arc is labeled $\twoheadrightarrow$ and its last arc is labeled $\to$, then $\mathbf{n}_1$ is a block of type $1$ and length $t$. Moreover, the ending vertex of that path must be the expansion $\mathbf{n}'_1\mathbf{n}_2\cdots\mathbf{n}_t$, with $\mathbf{n}'_1$ being the ordinary binary expansion of the number with expansion $n_1$. Since the property of having a maximal checking path of length $t\ge 3$, with the source of the graph as starting vertex, and such that its second arc is labeled $\twoheadrightarrow$ and its last arc is labeled $\to$, is clearly preserved under labeled graph isomorphisms, we conclude that under that hypothesis the statements~\textbf{S1} and~\textbf{S2} are satisfied. As explained before, this conclude the proof also in this case.

We are left with the case when there are no maximal checking paths with the mentioned property. The lack of such paths, together with the assumption that there are no arcs labeled $\to$ from the source $\mathbf{n}$, imply that $\mathbf{n}_1$ is a block of type $1$ and length $2$. Since these conditions are preserved under labeled graphs isomorphisms, this easily leads to the statement~\textbf{S1}, but does not imply the statement~\textbf{S2}. The problem is that the maximal checking path $\left(e_1,e_2\right)$ with $a=1$ (that is, that occurs on the first place) may correspond through $\varphi$ to a path made of a single-step reduction $f_1$ that occurs on a block of type $1$ and length $2$ in a different place, that is followed by a single-step reduction $f_2$ that occurs on an immediately subsequent block of type $2$. This possibility can be ruled out by noting that in this case from the middle vertex of $\left(f_1,f_2\right)$ one can certainly start a path $P$ in $A(m)$ such that for the (respective) images $f'_1$, $f'_2$ of $f_1$, $f_2$ via the place-preserving map through $P$, we have that $\left(f'_1,f'_2\right)$ is not the unique checking path of length $2$ that starts with $f'_1$. That is because the last single-step reduction on the block of type $1$ on which $f_1$ and $f'_1$ act, can be unlocked by $P$. This property cannot hold for $\left(e_1,e_2\right)$, unless it is not the unique checking path of length $2$ that starts with $e_1$, but this is impossible because $\left(f_1,f_2\right)$ does not enjoy this property, which is clearly preserved under graph isomorphisms. Since the aforementioned property that $\left(f_1,f_2\right)$ enjoys is also clearly preserved under graphs isomorphisms, this shows that the path $\left(e_1,e_2\right)$ must correspond to a checking path $\left(g_1,g_2\right)$ in $A(m)$, such that $g_1$ is a reduction that occurs on $\mathbf{m}_1=12$. But if $\mathbf{m}_2$ is of type $2$, $g_2$ may still occur on it. In this case, the middle vertex of $\left(g_1,g_2\right)$ is an expansion of the form $202\mathbf{x}$. The expansion $\mathbf{x}$ cannot begin with $2$, otherwise the checking path $\left(g_1,g_2\right)$ is not maximal as its corresponding $\left(e_1,e_2\right)$. If $\mathbf{x}$ is not the empty word, from $202\mathbf{x}$ one can certainly start a path $Q$ in $A(m)$ with ending vertex of the form $2022\mathbf{y}$. The checking path $\left(g'_1, g'_2\right)$ given by the respective images of $g_1$, $g_2$ via the place-preserving map through $Q$ is not maximal. But from the middle vertex of $\left(e_1,e_2\right)$ no paths with the same property of $Q$ can be constructed. To summarize, when $\mathbf{x}$ is not the empty word, then the ending vertices of the paths $\left(e_1,e_2\right)$ and $\left(g_1,g_2\right)$, that correspond via $\varphi$, are the expansions in the statement~\textbf{S2}, which is therefore true and prove the theorem.

To settle the last case, when $\mathbf{x}$ is empty, we can replace $\varphi$ with its composition (to the left) with the automorphism of $A(m)=A(10)$ that swaps the expansions
\[
210\qquad\text{and}\qquad 1002
\]
and fixes $122$, $202$ and $1010$. With this replacement, the statement~\textbf{S2} which we dreamed at comes true, and the theorem is proved.
\end{proof}

The following theorem is an immediate consequence, if one takes into account that $A(n)$ and $A(2n+1)$ are isomorphic for any nonnegative integer $n$.

\begin{thm}
The graphs $A(m)$ and $A(n)$ are isomorphic as edge-labeled directed graphs if and only if there exists a nonnegative integer $t$ such that
\[
m=2^tn+2^t-1\quad\lor\quad n=2^tm+2^t-1\;.
\]
\end{thm}

\section{An iterative formula for \texorpdfstring{$b(n)$}{b(n)}}

\begin{rem}\label{Calc}
In notation of \autoref{CG}, there is a canonical bijection between $\mathcal{H}(n)$ and the subset $H\subseteq\mathcal{H}(n_1)\times\mathcal{H}(n_2)$ made of all pairs $\left(\mathbf{n}''_1,\mathbf{n}'_2\right)$ such that $\mathbf{n}''_1$ ends with $0$ or $\mathbf{n}'_2$ is a short or empty expansion. Hence, denoting by $b_0$, $b_2$ the number of expansions of $n_1$ that end, respectively, with $0$ and with $2$, and by $s$ the number of short or empty expansions of $n_2$, we have
\[
b(n)=b_0b\left(n_2\right)+b_2s\;.
\]
In particular, when $\mathbf{n}_1$ is a block of length $a$, we have
\[
b(n)=\left\{\begin{array}{lr}ab\left(n_2\right)+s&\text{if $\mathbf{n}_1$ is of type $1$}\\&\\b\left(n_2\right)+as&\text{if $\mathbf{n}_1$ is of type $2$}\end{array}\right.\;;
\]
moreover, the number of short expansions of $n$ is
\[
\left\{\begin{array}{lr}(a-1)b\left(n_2\right)+s&\text{if $\mathbf{n}_1$ is of type $1$}\\&\\s&\text{if $\mathbf{n}_1$ is of type $2$}\end{array}\right.\;.
\]
\end{rem}

\begin{prop}
Let $\mathbf{n}=\mathbf{n}_1\cdots\mathbf{n}_r$ be the block decomposition of the minimal hyperbinary expansion of an even nonnegative integer $n$, $\left(a_1,\ldots,a_r\right)$ the sequence of the (respective) lengths of the blocks and $\left(t_1,\ldots,t_r\right)$ the sequence of $1$s and $2$s given by their types.

Let us also define the following functions on triples of integers
\[
b_1(\alpha,\beta,\sigma):=\alpha\beta+\sigma\;,\qquad b_2(\alpha,\beta,\sigma):=\beta+\alpha\sigma\;,
\]
\[
s_1(\alpha,\beta,\sigma):=(\alpha-1)\beta+\sigma\;,\qquad s_2(\alpha,\beta,\sigma):=\sigma\;.
\]
Then we have
\[
b(n)=h_r\;,
\]
with $h_i$, $i\in\{0,\ldots, r\}$, being recursively defined (together with the auxiliary numbers $k_i$) by
\[
h_0:=1\;,\quad k_0:=1\;,
\]
\[
h_{i+1}:=b_{t_{r-i}}\left(a_{r-i},h_i,k_i\right)\;,\quad k_{i+1}:=s_{t_{r-i}}\left(a_{r-i},h_i,k_i\right)\;,\qquad\qquad i\in\{0,\ldots, r-1\}\;.
\]
\end{prop}
\begin{proof}
By induction on $r$, it is a straightforward consequence of \autoref{Calc}.
\end{proof}
The above recipe may also be encapsulated in the following (a bit cumbersome) formula:

\begin{multline*}
b(n)=b_{t_1}\Big(\;a_1,b_{t_2}\big(\ldots\left(a_{r-1},b_{t_r}(a_r,1,1),s_{t_r}(a_r,1,1)\right)\ldots\big)\,,\\s_{t_2}\big(\ldots\left(a_{r-1},\;b_{t_r}(a_r,1,1),\;s_{t_r}(a_r,1,1)\right)\ldots\big)\;\Big)\;.
\end{multline*}

To turn the recipe into an algorithm with input the binary representation $d_t\cdots d_0$ of $n=\sum_{\ell=0}^td_\ell2^\ell$ is quite straightforward (one has to consider the binary expansions $\mathbf{n}'_i$s that come from the blocks; $a_1$ and $a_2$ below count the length of the current block, according to its type):

\phantomsection\label{Alg}\begin{algorithm}[H]\caption{$b(n)$ with input the binary expansion $d_t\cdots d_0$~of~$n$}
	\begin{algorithmic}[0]
		\Function{$b$}{$n$}
        \State $i_0\gets 0$
        \While {$d_{i_0}=1$}
        \State $i_0\gets i_0+1$ \Comment Disregard the tail of $1$s
        \EndWhile
        \State $i_0\gets i_0+1$ \Comment Avoid a step for which $d_{i_0}$ is necessarily $0$
		\State $a_1\gets 0$; $a_2\gets 0$; $b \gets 1$; $s\gets 1$
		\For{$\ell \gets i_0, t$}
		\If{$d_\ell=1$}
        \If{$a_1=0$}\State $a_2\gets a_2+1$
        \Else\State $s\gets a_1\cdot b+s$; $b\gets b+s$; $a_1\gets 0$
        \EndIf
		\Else
        \If{$a_2=0$}\State $a_1\gets a_1+1$
		\Else\State $b\gets b+a_2\cdot s$; $a_2\gets 0$; $a_1\gets 1$	
		\EndIf
		\EndIf
		\EndFor
        \State $b\gets b+a_2\cdot s$
		\State \textbf{return} $b$
		\EndFunction
	\end{algorithmic}
\end{algorithm}

Note that the algorithm is iterative, hence much faster than a direct implementation of the well-known formula $b(2n)=b(n)+b(n-1)$, $b(2n+1)=b(n)$ as a recursive function. But, as J.\ Shallit has pointed out to us, to make iterative such a formula is also an easy application of the theory of $k$-regular sequences, or more generally of memoization (and therefore is basically a matter of dynamic programming). In a nutshell, looking at the column pair \[\mathbf{b}(n):=\left(\begin{array}{c}b(n)\\b(n-1)\end{array}\right)\;,\]
for $n>0$ we have
\[
\mathbf{b}(2n)=\left(\begin{array}{cc}1&1\\0&1\end{array}\right)\mathbf{b}(n)
\]
and
\[
\mathbf{b}(2n+1)=\left(\begin{array}{cc}1&0\\1&1\end{array}\right)\mathbf{b}(n)\;.
\]
Hence, denoting respectively by $M_0$ and $M_1$ the above $2\times2$ matrices, $b(n)$ is the top entry of
\begin{equation}\label{Easy}
M_{d_0}\cdots M_{d_t}\left(\begin{array}{c}1\\0\end{array}\right)\;.
\end{equation}
The above formula can also be arranged in terms of blocks of $0$s and of $1$s, by exploiting
\begin{equation}\label{Power}
{M_0}^a=\left(\begin{array}{cc}1&a\\0&1\end{array}\right)\;,\qquad {M_1}^a=\left(\begin{array}{cc}1&0\\ a&1\end{array}\right)\;.
\end{equation}
The algorithm that arises in this way scans the digits left to right, meanwhile \hyperref[Alg]{Algorithm 1} works right to left. But since the top entry in \autoref{Easy} is given by multiplication to the left by the one-row matrix $(1,0)$, one may also immediately deduce a right to left algorithm.

Such kind of algorithms are addressed in the work \cite{M}, too. Actually, \cite[Algorithm~1]{M} calculates the number of BSD representations of $n$ on $i$ ternary units of information, taking as input the binary representation $n_{i-1}\cdots n_0$ of $n$ with length $i$ (possibly with some initial $0$s). To get $b(n)$, according to \cite[Theorem~4]{M}, it suffices to swap $0$s and $1$s in that algorithm. This way one gets exactly an implementation of \autoref{Easy}. In this respect, what we find interesting is the explanation given in \cite{M}, which differs from the one we provided for \autoref{Easy}, which may be considered more standard. After the preliminary stage, where $n$ is replaced with a smaller number $m$ on which the function (the number of BSD representations) takes the same value, the first iteration of \cite[Algorithm~1]{M} looks at the values on the endpoints of the interval $\left[0,2^j\right]\ni m$. These values are named $\mathit{lower}$ and $\mathit{upper}$, respectively. Then the first digit $n_{i-1}$ blesses the subinterval among $\left[0,2^{j-1}-1\right]$ and $\left[2^{j-1},2^j\right]$ which $m$ belongs to. Then $\mathit{lower}$ and $\mathit{upper}$ are replaced with the values relative to that subinterval, simply by taking a sum (due to \cite[Theorem~7]{M}, which can be seen as a `BSD-version' of the formula $b(2n)=b(n)+b(n-1)$, if $b(2n+1)=b(n)$ is taken as granted). The iteration proceeds in the same way, till the requested value on $m$ (hence on $n$). Though the resulting algorithm is the same as the one provided by the more standard setting, the description as a kind of divide and conquer strategy (where one of the two parts is simply disregarded) is valuable in our opinion, because might help to handle possible different situations in which a standard approach may not work. In addition, \cite[Algorithm~1]{M} opens the way to \cite[Algorithm~2]{M} for the function \texttt{BSD-\scriptsize FAST}. The enhancement principally relies on the length $i$ of the BSD representations of $n$, which is substituted by $k:=\lceil\log_2 n\rceil$ for $n<2^i$. When the method is applied to the calculation of $b(n)$ (by exchanging $0$s with $1$s), the enhancement consists in working out at once the first block of $1$s in the ordinary binary representation of $n$, by means of a multiplication by the length $i-k$ of that block. In terms of the description given by \autoref{Easy}, the enhancement consists in the substitution of the tail
\[
M_{d_{i-k}}\cdots M_{d_{i-1}}\left(\begin{array}{c}1\\0\end{array}\right)
\]
with
\[
M_1^{i-k}\left(\begin{array}{c}1\\0\end{array}\right)=\left(\begin{array}{c}1\\ i-k\end{array}\right)\;,
\]
followed by the column splitting $(1,i-k)=(1,0)+(i-k)(0,1)$. In these terms, the call $\texttt{BSD}(n,k)+(i-k)\times\texttt{BSD}(2^k-n,k)$ in \cite[Algorithm~2]{M}, where \texttt{BSD} refers to \cite[Algorithm~1]{M}, duplicates the multiplication by $M_{d_0}\cdots M_{d_{i-k-1}}$. Anyway, \cite[Algorithm~2]{M} gives a relevant improvement when $k<<i$, because is $\mathcal{O}(\log n)$ under the hypothesis that the calculation of $k$, which amounts to counting the $1$s in the first block, is negligible. It seems to us that a further enhancement would be obtained by replacing the mentioned call with a recursive one: $\texttt{BSD}(n,k)+(i-k)\times\texttt{BSD-\scriptsize FAST}(2^k-n,k)$. On the other hand, by the same reasons, when \autoref{Easy} is arranged block-wise by means of \autoref{Power}, that approach seems to give an even faster algorithm because it does not involve the mentioned duplication of \cite[Algorithm~2]{M}.

To compare with \hyperref[Alg]{Algorithm 1} here, note that there is a slight difference in the block subdivision. Let us consider for instance $n=42$: its hyperbinary expansion $12122$ involve only three blocks, meanwhile one cannot take advantage of \autoref{Power} when using \autoref{Easy}, as the binary representation is $101010$. This seems to give an advantage to \hyperref[Alg]{Algorithm 1}, under the hypothesis that the instructions that increase or reset a counter, $a_1$ or $a_2$, can be neglected for the purposes of time complexity determination. Namely, when running \hyperref[Alg]{Algorithm 1} on the input $n=42$, the expensive instructions (the ones in the else clauses) are executed only three times.

As pointed out in the introduction, mainly due to the striking result of \cite{CW}, the connection with the Stern's diatomic sequence $c(n)=b(n-1)$ may be of particular interest (that is why, like in \cite{M}, we explicitly mentioned that sequence in the title). J. Shallit has also pointed out to us that a matrix description which directly works for $c$, is obtained by setting
\[
M_0:=\left(\begin{array}{cc}1&0\\1&1\end{array}\right)\;,\qquad M_1:=\left(\begin{array}{cc}0&1\\-1&2\end{array}\right)\;.
\]

A thorough comparison of computational costs may be quite subtle and go beyond the author's skills, both in the theoretical and in the applied fields. For instance, in the above outline, though operations on counters in \hyperref[Alg]{Algorithm 1} may be somewhat assimilated to the calculation of $k$ in \cite[Algorithm~2]{M}, it can be argued that in the latter case to detect a single block of $1$s is an operation that in many cases can be performed at CPU level. It may also be argued that the algorithms we are dealing with are all $\mathcal{O}(\log n)$, hence to compare them may appear of little interest. But, like in the case of the interesting description of \cite[Algorithm~1]{M}, we believe that to explicitly connect the earlier discrete-geometric results with their related computational aspects that are discussed in this section, may turn to be useful in different situations of the same kind.

\end{document}